\documentclass{amsart}
\usepackage{amssymb}
\usepackage{enumerate}
\usepackage{amsmath}
\usepackage{mathrsfs}
\usepackage{tikz}
\usetikzlibrary{matrix,shapes,arrows,fit}
\usepackage{float}
\usepackage{physics}

\newtheorem{theorem}{Theorem}
\newtheorem{proposition}[theorem]{Proposition}
\newtheorem{lemma}[theorem]{Lemma}

\numberwithin{theorem}{section}

\theoremstyle{definition}
\newtheorem{definition}[theorem]{Definition}
\newtheorem{example}[theorem]{Example}

\theoremstyle{remark}

 \newcommand{\RacahOne}[6]{
  \begin{tikzpicture}
 	\draw (-0.2,-0.1)--(-3,-0.1);
	\draw (-3,-0.1)--(-1.6,-2.5);
	\draw (-1.6,-2.5)--(-0.2,-0.1);
	\draw (-0.8,-0.4) node {$#3$};
	\draw (-1.6,-0.4) node {$#2$};
	\draw (-2.4,-0.4) node {$#1$};
	\draw (-1.2,-1.0) node {$#5$};
	\draw (-2,-1.0) node {$#4$};
	\draw (-1.6,-1.6) node {$#6$};
 \end{tikzpicture}
 }

\newcommand{\Zn}{(\mathbb{Z}_2)^n}
\newcommand{\su}{\mathfrak{su}(1,1)} 
\newcommand{\usu}{\mathcal{U}(\mathfrak{su}(1,1))}

\begin{document}

\title{A discrete realization of the higher rank Racah algebra}

\author{Hendrik De Bie}
\email{Hendrik.DeBie@UGent.be}
\address{Department of Mathematical Analysis\\Faculty of Engineering and Architecture\\Ghent University\\Building S8, Krijgslaan 281, 9000 Gent\\ Belgium.}



\author{Wouter van de Vijver}
\email{Wouter.vandeVijver.be}
\address{Department of Mathematical Analysis\\Faculty of Engineering and Architecture\\Ghent University\\Building S8, Krijgslaan 281, 9000 Gent\\ Belgium.}

\date{\today}
\keywords{Racah algebra, Racah polynomials, superintegrable system, coupling coefficients, difference operators}
\subjclass[2010]{33C50, 33C80; 47B39; 81R10, 81R12}

\begin{abstract}
In previous work a higher rank generalization $R(n)$ of the Racah algebra was defined abstractly. The special case of rank one 
 encodes the bispectrality of the univariate Racah polynomials and is known to admit an explicit realization in terms of the operators associated to these polynomials.
 Starting from the Dunkl model for which we have an action by $R(n)$ on the Dunkl-harmonics, we show that connection coefficients between bases of Dunkl-harmonics diagonalizing certain Abelian subalgebra are multivariate Racah polynomials. By lifting the action of $R(n)$ to the connection coefficients, we identify the action of the Abelian subalgebras with the action of the Racah operators defined by J. S. Geronimo and  P. Iliev. Making appropriate changes of basis one can identify each generator of $R(n)$ as a discrete operator acting on the multivariate Racah polynomials.
\end{abstract}

\maketitle

\section{Introduction}
\setcounter{equation}{0}

The Racah polynomials form a family of discrete orthogonal polynomials of ${}_{4}F_{3}$ hypergeometric type, that sits at the top of one branch of the celebrated Askey scheme \cite[Chapter~9]{Koekoek&Lesky&Swarttouw-2010} of orthogonal polynomials.  
Their properties can be explained by making use of the the so-called Racah algebra introduced in \cite{Granovskii&Zhedanov-1988}. It is defined as the infinite-dimensional associative algebra over $\mathbb{C}$ with generators $K_1$ and $K_2$ that satisfy, together with their commutator $K_3=\left[K_1,K_2\right]$, the following relations:
\begin{align}\label{eq:Racah}
\begin{split}
\left[K_2,K_3\right]&=K_2^2+\left\{K_1,K_2\right\}+dK_2+e_1, \\
\left[K_3,K_1\right]&=K_1^2+\left\{K_1,K_2\right\}+dK_1+e_2,
\end{split}
\end{align}
with $\left\{A,B\right\}:=AB+BA$ and where $d$, $e_1$, $e_2$ are structure constants. A realization of this algebra in terms of difference operators then encodes the (bispectral) properties of the Racah polynomials. The Racah algebra can alternatively be presented in terms of generators  $C_{12}$, $C_{23}$, $C_{13}$ and $F$. In this form one has $C_{12}+C_{23}+C_{13}=G$ and $\left[C_{23},C_{13}\right]=\left[C_{13},C_{12}\right]=\left[C_{12},C_{23}\right]=2F$, where $G$ is a central operator. The relations then read 
\begin{align} \label{equit}
\begin{split}
\left[C_{12},F\right]&=C_{23}C_{12}-C_{12}C_{13}+i_{12},\\
\left[C_{23},F\right]&=C_{13}C_{23}-C_{23}C_{12}+i_{23}, \\
\left[C_{13},F\right]&=C_{12}C_{13}-C_{13}C_{23}+i_{13},
\end{split}
\end{align}
with $i_{23}, i_{13}, i_{12}$ structure constants (or central elements). The relations (\ref{equit}) appeared initially in \cite{Gao&Wang&Hou-2013}. More details on the passage between \eqref{eq:Racah} and \eqref{equit} can be found in \cite{Genest&Vinet&Zhedanov-2014}. Note that in the explicit difference operator realization, one of the operators $C_{23}$, $C_{13}$, $C_{12}$ then acts diagonally on all Racah polynomials. This is shown in the subsequent Proposition \ref{DiscreteRankOne}.
The presentation (\ref{equit}) is commonly referred to as the ``equitable'' or ``democratic'' presentation \cite{Levy-Leblond&Levy-Nahas-1965,  Terwilliger-2006}. It is precisely this shape that will play a crucial role in the present paper.

The Racah algebra has deep connections with superintegrable systems and with the 3-fold tensor product of $\usu$, the universal enveloping algebra of the Lie algebra $\su$.

We call a quantum system with Hamiltonian $H$ and $d$ degrees of freedom maximally superintegrable if it has $2d-1$ algebraically independent constants of motion (including $H$). We consider models, where the conserved charges are of degree at most $2$ in the momenta. In 2 dimensions they can all be obtained as limits of the so-called generic $3$-parameter system on the $2$-sphere \cite{Kalnins&Miller&Post-2013}. Its Hamiltonian is given by:
\begin{equation} \label{eq:Hamiltonian1}
 H:=\sum_{1\leq i < j \leq 3} J_{ij}^2
 +(x_1^2 + x_2^2 + x_{3}^2)\sum_{i=1}^{3}\frac{a_i}{x_i^2},
\end{equation}
where $J_{ij}=i(x_j \partial_{x_i} - x_i \partial_{x_j})$ are the angular momentum operators and where $a_1,a_2,a_3$ are real parameters.  
The Racah algebra surfaced in this context when the symmetry algebra formed by the constants of motion of $H$ was realized in \cite{Kalnins&Miller&Post-2007} using the difference operator of which the Racah polynomials are eigenfunctions.  In \cite{Genest&Vinet&Zhedanov-2014-2} this observation was cast in the framework of the Racah problem for $\mathfrak{su}(1,1)$. Indeed, upon using the singular oscillator realization of $\mathfrak{su}(1,1)$, the Hamiltonian $H$ can be identified with the total Casimir operator obtained when three such representations are combined. The intermediate Casimir operators then satisfy the relations of the (centrally extended) Racah algebra (\ref{equit}). For a review we refer the reader to \cite{Genest&Vinet&Zhedanov-2014-3}.  



It thus becomes very natural to extend the Racah algebra to higher rank. This was initially done in the superintegrable context \cite{Iliev-2017} using the $n$-dimensional generalization of (\ref{eq:Hamiltonian1}), and the irreducibility of the representation spaces that arise in this model has been proven in \cite{Iliev-2018}. In the context of abstract tensor products, the construction was given in \cite{DeBie&Genest$vandeVijver&Vinet} and illustrated with a new model using the $\Zn$ Laplace-Dunkl operator \cite{DTAMS, DX}. 
The extension goes as follows. Consider now the $n$-fold tensor product of $\usu$ and the Casimir $C$ of $\su$. Using the coproduct, $C$ can be lifted to an arbitrary number of factors in the tensor product, leading to intermediate Casimirs $C_A$ defined for any subset $A \subset [n]=\{1, \ldots, n\}$. Here the set $A$ describes the relevant factors in the tensor product. The algebra generated by $\{ C_A\}$ is then called the Racah algebra of rank $n-2$.

A multivariate generalization of the Racah polynomials has been defined, see \cite{Tratnik-1991}. A natural problem is hence to find the connection between the higher rank Racah algebra on the one hand and the multivariate Racah polynomials on the other hand. In the present paper we will solve that problem by constructing an explicit realization of the Racah algebra in terms of difference operators acting on the Racah polynomials.

For the solution of this problem, we will benefit from two crucial facts from the existing literature. First, there already exists a realization of a distinct Abelian subalgebra of the higher rank Racah algebra on the multivariate Racah polynomials. This was achieved through a complicated construction in \cite{Geronimo&Iliev-2010}, with a different aim, namely to investigate the multispectrality of these polynomials. Second, the multivariate Racah polynomials arise as coupling coefficients between different bases of representations of $n$-fold tensor products of $\usu$, as shown in \cite{DeBie&Genest$vandeVijver&Vinet}.

Our construction can then be outlined as follows.
The action of the Racah algebra in the Dunkl realization of  \cite{DeBie&Genest$vandeVijver&Vinet} can be lifted to an action on the coupling coefficients arising in this model, which we know are multivariate Racah polynomials. The two bases related by the coupling coefficients are diagonalizing two distinct Abelian subalgebras. Lifting the action of these Abelian subalgebras to the connection coefficients, one arrives at the conclusion that one Abelian subalgebra is realized by operators that are multiplication by certain quadratic polynomials. The other Abelian algebra has the multivariate Racah polynomials as eigenvectors. Then one concludes that this Abelian algebra coincides with the realization given in \cite{Geronimo&Iliev-2010}. The full discrete realization of the Racah algebra is then obtained by making changes of basis to bases diagonalizing different Abelian subalgebras.

The paper is organized as follows. In Section \ref{Preliminaries} the higher rank Racah algebra is introduced and the properties needed for the article are given. The univariate and multivariate Racah polynomials are defined as well as the operators $\mathcal{L}_i$ acting on these polynomials. In Section \ref{TheMainTheorem} we state the main Theorem. We present two examples of this theorem, the rank one and rank two case.
In Section \ref{MultivariateRacahPolynomialsAsConnectionCoefficients} we recapitulate the Dunkl model and show that the multivariate Racah polynomials appear as connection coefficients between bases of Dunkl harmonics. We finish this section with a discussion on the gauge coefficients appearing along the multivariate Racah polynomials. In Section \ref{DiscreteRealization} we present the proof for the main Theorem. We first give a proof of the rank two case before finally tackling the general case.

\section{Preliminaries}\label{Preliminaries}

\subsection{The Racah algebra}
The algebra $\mathfrak{su}(1,1)$ is generated by three elements $J_\pm$ and $A_0$ with well-known relations:
\begin{equation}\label{su11}
[J_-,J_+]=2A_0, \qquad [A_0, J_\pm]=\pm J_\pm.
\end{equation}
The universal enveloping algebra $\mathcal{U}(\mathfrak{su}(1,1))$ contains the Casimir element of $\mathfrak{su}(1,1)$:
\begin{equation}\label{Casimir}
	C:=A_0^2-A_0-J_+J_-.
\end{equation}
There exists a morphism $\mu^*$ called the comultiplication that maps $\mathfrak{su}(1,1)$ into $\mathfrak{su}(1,1)\otimes\mathfrak{su}(1,1)$ with following action on the generators:
\begin{align*}
 \mu^*(J_{\pm})=J_{\pm}\otimes 1 + 1 \otimes J_{\pm}, \\
  \mu^*(A_{0})=A_{0}\otimes 1 + 1 \otimes A_{0}.
\end{align*}
This map extends to the universal enveloping algebra. We are now able to define the operators that generate the higher rank Racah algebra. Define inductively
$$
\mathscr{C}_1:=C, \qquad  \mathscr{C}_n:=(\underbrace{1\otimes\ldots\otimes 1}_{n-2 \text{ times }} \otimes \mu^*)(\mathscr{C}_{n-1}).
$$
 Consider the map
$$
\tau_k: \bigotimes_{i=1}^{m-1} \mathcal{U}(\mathfrak{su}(1,1))\rightarrow \bigotimes_{i=1}^{m} \mathcal{U}(\mathfrak{su}(1,1)),
$$
which acts as follows on homogeneous tensor products:
$$
\tau_k(t_1\otimes \ldots \otimes t_{m-1}):=t_1\otimes \ldots \otimes t_{k-1} \otimes 1 \otimes t_k \otimes \ldots \otimes t_{m-1} .
$$
and extend it by linearity.  The map $\tau$ adds a $1$ at the $k$-th place. This allows to define the following:
\begin{align}
\label{Casimir-Upper}
C_A:=\left(\prod_{k \in \left[ n\right] \backslash A}^{\longrightarrow} \tau_k \right)\left(\mathscr{C}_{|A|}\right).
\end{align}
with $A$ a subset of $[n]$. Here we introduce the notation $[n]:=\{1,\ldots, n \}$ for the set of integers from $1$ to $n$. Later we will also use the notation $[i \ldots n]:=\{i,\ldots ,n \}$.
\begin{definition}
The Racah algebra of rank $n-2$ is the subalgebra of $\bigotimes_{i=1}^{n} \mathcal{U}(\mathfrak{su}(1,1))$ generated by the set $\{ C_A | A \subset [n] \}$. We denote this algebra by $R(n)$.
\end{definition}
For a full account on this generalized Racah algebra $R(n)$ we refer the reader to \cite{DeBie&Genest$vandeVijver&Vinet}. For our purposes we only need to know the following three facts, which we collect in three lemmas. 

\begin{lemma}\label{Commute}
 If either $A \subset B$ or $B \subset A$ or $A \cap B=\emptyset$ then $C_A$ and $C_B$ commute.
\end{lemma}
This follows by construction of the operators $C_A$. Observe that by this lemma $C_i$ and $C_{[n]}$ are central in $R(n)$. Consider a chain of sets $A_1 \varsubsetneq A_2 \varsubsetneq \dots A_j$, each set strictly contained in the next one. The operators belonging to these sets $C_{A_i}$ generate an Abelian algebra. 
\begin{definition}\label{labelling}
Consider a chain $\mathcal{C}$ with maximal length: $A_i \varsubsetneq A_{i+1}$ with $|A_{i+1}|=|A_{i}|+1$. The set of Casimirs belonging to this chain excluding the central Casimirs generate the Abelian algebra
\[
 \mathscr{Y}_{\mathcal{C}}=\langle C_{B}|B \in \mathcal{C} \text { and } |B|\neq 1, n \rangle.
\]
We call this algebra a labelling Abelian algebra.
\end{definition}
\noindent
The rank of a Racah algebra is equal to the dimension of a labelling Abelian algebra.

The operators $C_A$ are not linearly independent:
\begin{lemma}\label{Lind} For any set $A \subset [n]$, it holds that
\begin{equation}
 C_A=\sum_{\left\{i,j\right\}\subset A} C_{ij}-\left(|A|-2\right)\sum_{i \in A} C_i .
\end{equation}
\end{lemma}
Here we abbreviate the notation $\{i,j\}$ by $ij$ in the index of $C_{ij}$.
This means that the set $\mathcal{B}_1= \{ C_{ij}: \{i,j\} \subset [n] \} \cup \{ C_{i}| i \in  [n] \}$ generates $R(n)$.
Also the set $\mathcal{B}_2=\{ C_{[i \ldots j]}| i \leq j\}$ generates the full algebra $R(n)$ as every element of $\mathcal{B}_1$ can be written as  a linear combination of $\mathcal{B}_2$ in the following way by Lemma \ref{Lind}:
\[
 C_{ij}=C_{[i\ldots j]}-C_{[i+1 \ldots j]}-C_{[i \ldots j-1]}+C_{[i+1 \ldots j-1]}+C_i+C_j.
\]

The third lemma is the most important one:
\begin{lemma}\label{Racah1}
Let $K$, $L$ and $M$ be three disjoint subsets of $[n]$. The algebra generated by the set
\[ \{ C_K,C_L,C_M,C_{K \cup L}, C_{K \cup M}, C_{ L \cup M}, C_{K \cup L \cup M}  \} \]
is isomorphic to the rank $1$ algebra $R(3)$ and has therefore the following relations:
Introduce the operator $F$:
\begin{align}\label{B}
2F:=[C_{KL},C_{LM}] = [C_{KM},C_{KL}] = [C_{LM},C_{KM}].
\end{align}
Together with the previous lemmas the relations become
\begin{align} \label{RankoneRacah}
\begin{split} \\
[C_{KL},F]&=C_{LM}C_{KL}-C_{KL}C_{KM}+\left(C_L-C_K\right)\left(C_{M}-C_{KLM}\right), \\
[C_{LM},F]&=C_{KM}C_{LM}-C_{LM}C_{KL}+\left(C_M-C_L\right)\left(C_{K}-C_{KLM}\right),\\
[C_{KM},F]&=C_{KL}C_{KM}-C_{KM}C_{LM}+\left(C_K-C_M\right)\left(C_{L}-C_{KLM}\right).
\end{split}
\end{align}
We denote this algebra by $R^{K,L,M}(3)$.
\end{lemma}
The rank one Racah algebra $R(3)$ is a central extension of the classical Racah algebra. This can be seen by considering relations (\ref{B}) and (\ref{RankoneRacah}). Throughout this paper we will often make use of Lemma \ref{Racah1} and properties of the rank one Racah algebra. The next subsection will give an overview of these properties.

\subsection{The rank one Racah algebra}
We restrict our attention to the rank one case, the algebra $R(3)$. This algebra is generated by the following set of operators:
\[
  \{ C_1,C_2,C_3,C_{12}, C_{13}, C_{23}, C_{123} \}.
\]
By Lemma \ref{Commute} the operators $C_1$, $C_2$ and $C_3$ and $C_{123}$ are central. By Lemma \ref{Lind} they have the following linear dependency
\begin{equation}\label{Lind1}
 C_{123}=C_{12}+C_{13}+C_{23}-C_1-C_2-C_3.
\end{equation}
We will represent this algebra by the following diagram. We left out $C_{13}$ as it is a linear combination of the other generators. 
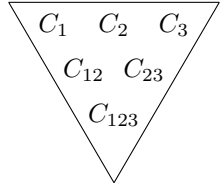
\begin{center}
 \RacahOne{C_1}{C_2}{C_3}{C_{12}}{C_{23}}{C_{123}}
 \end{center}
Let $V$ be an irreducible representation of this algebra such that $C_{12}$ and $C_{23}$ can be diagonalized. The central elements will act as scalars. 
Consider Lemma \ref{Racah1} with $K=\{1\}$, $L=\{2\}$ and $M=\{3\}$ to see that $C_{12}$ and $C_{23}$ do not commute and therefore cannot be diagonalized by the same basis. Let $\psi_\ell$ be a basis diagonalized by $C_{12}$ and $\varphi_k$ be a basis diagonalized by $C_{23}$. An interesting question arises: 
What are the connection coefficients between these bases? It turns out these are the Racah polynomials

\begin{definition} (\cite{Koekoek&Lesky&Swarttouw-2010})
Let $r_n(\alpha,\beta,\gamma, \delta; x)$ be the classical univariate Racah polynomials
 \begin{align*}
  	r_n(\alpha,\beta,\gamma, \delta; x)&:=(\alpha+1)_n(\beta+\delta+1)_n(\gamma+1)_n \\
  	& \quad \times {}_4F_3 \left[ \substack{ -n, n+\alpha+\beta+1, -x, x+\gamma+\delta+1 \\ \alpha+1, \beta+\delta+1, \gamma+1}; 1 \right].
 \end{align*}
\end{definition}

We have the following proposition (see also  \cite{Kalnins&Miller&Post-2011},\cite{Post}):
\begin{proposition}\label{UnivariateRacah}
Let $V$ be an irreducible N-dimensional representation of $R(3)$. Let $\{ \psi_\ell\}$ be a basis that diagonalizes $C_{12}$ and $\{ \varphi_k\}$ be a basis that diagonalizes $C_{23}$. The connection coefficients $R_k(\ell)$ given by
\[ \sum_\ell R_k (\ell) \psi_\ell=\varphi_k\]
are up to a gauge factor the classical univariate Racah polynomials. 
\end{proposition}

The explicit form of the Racah polynomials can be determined by the action of the generators of $R(3)$.
For that aim we introduce the polynomial 
\begin{equation}\label{kappa}
\kappa(x,\beta):=\left(x+\frac{\beta+1}{2}\right)\left(x+\frac{\beta-1}{2}\right).
\end{equation}
The eigenvalues of the operators take the following form:
\begin{align}\label{spectra1}
\begin{split}
   C_1&:\kappa(0,\beta_0) \\ 
   C_2&:\kappa(0,\beta_1-\beta_0-1) \\
   C_3&:\kappa(0,\beta_2-\beta_1-1) \\
   C_{12}&: \kappa(\ell,\beta_1) \\
   C_{23}&:\kappa(k,\beta_2-\beta_0-1) \\
   C_{123}&:\kappa(N,\beta_2).
   \end{split}
\end{align}
The constants $\beta_0$, $\beta_1$ and $\beta_2$ are parameters depending on the representation of $R(3)$. They have been cast in a form that suits our needs best.
Then the connection coefficients are given up to a gauge factor by
\begin{equation}\label{ConnRacah}
  \hat{R}_k(\ell):=r_{k}(\beta_1-\beta_0-1,\beta_{2}-\beta_1-1,-N-1, \beta_1+N;\ell).
\end{equation} 
The number $\hat{R}_k(\ell)$ differs from $R_{k}(\ell)$ by the gauge factor. The gauge factor can be found by normalizing both bases. If this is the case, the connection coefficients are normalized Racah polynomials. Moreover, these polynomials are orthogonal for the following weight:
\begin{equation}\label{ConnRho}
	 \rho(\ell)=\frac{\Gamma(\beta_1+\ell)\Gamma(\beta_1-\beta_0+\ell)}{\Gamma(\beta_0+1+\ell)\ell!}(\beta_1+2\ell)\frac{\Gamma(\beta_2+N+\ell)\Gamma(\beta_2-\beta_1+N-\ell)}{\Gamma(\beta_1+N+1+\ell)(N-\ell)!} .
\end{equation}
We have
\begin{equation} \label{ConnOrtho}
	\sum_{\ell=0}^{N} \frac{\rho(\ell)}{\lambda(k)}\hat{R}_k(\ell)\hat{R}_{k'}(\ell)=\delta_{k,k'} 
\end{equation}
with $\frac{1}{\lambda(k)}$ equal to
\begin{equation}\label{ConnLAmbda} 
\frac{\Gamma(k+\beta_2-\beta_0-1)}{\Gamma(k+\beta_2-\beta_1)\Gamma(k+\beta_1-\beta_0)k!}(2k+\beta_2-\beta_0-1)\frac{\Gamma(\beta_0+N+1-k)(N-k)!}{\Gamma(k+\beta_2-\beta_0+N)\Gamma(\beta_2+N+k)} .
\end{equation}

The relation between Racah polynomials and $R(3)$ can also be given more explicitly
\begin{proposition}\label{DiscreteRankOne}
Let $\hat{C}_{12}=\kappa(x,\beta_1)$ and $\hat{C}_{23}=-\mathcal{L}_1+\kappa(0,\beta_2-\beta_0-1)$ with  $\mathcal{L}_1$ a shift operator explicitly given by
 
\begin{align*}
 \mathcal{L}_1=\frac{(x+\beta_1-\beta_0)(x+\beta_1)(N+x+\beta_2)(N-x)}{(2x+\beta_1)(2x+\beta_1+1)}(E_x-1)+ \\
 \frac{(x+\beta_0)x(N-x-\beta_1+\beta_2)(N+x+\beta_1)}{(2x+\beta_1)(2x+\beta_1-1)}(E_x^{-1}-1)\\
 \end{align*}
with the operator $E_x(f(x)):=f(x+1)$. The central operators $\hat{C}_1$, $\hat{C}_2$, $\hat{C}_3$ and $\hat{C}_{123}$ are constants as given in formula (\ref{spectra1}) and the operator $\hat{C}_{13}$ is given by formula (\ref{Lind1}).
 They generate a discrete realization of the $R(3)$. The eigenvectors of $\hat{C}_{23}$ are exactly the univariate Racah polynomials 
\[
r_{k}(\beta_1-\beta_0-1,\beta_{2}-\beta_1-1,-N-1, \beta_1+N;x)
\]
with corresponding eigenvalues $\kappa(k,\beta_2-\beta_0-1)$.
\end{proposition}
\begin{proof}
 After explicit computation one concludes that this is indeed a discrete realization. Note that $C_{23}$ is in fact the three-point difference equation of the Racah polynomials (see also \cite{Genest&Vinet&Zhedanov-2014-3, Koekoek&Lesky&Swarttouw-2010}).
\end{proof}

Generalization of the operator $\mathcal{L}_1$ are given in the following section.

\subsection{Multivariate Racah polynomials and the operators $\mathcal{L}_i$}
M.V. Tratnik generalized the univariate Racah polynomials in \cite{Tratnik-1991}. We will use the definition given in \cite{Geronimo&Iliev-2010}.
\begin{definition}\label{defRacah}
Let $p\leq n$. The multivariate Racah polynomials are given by 
\begin{align*}
 	R_{p}(\vec k; \vec x; \vec \beta ;N) 
 	&=\prod_{j=1}^{p} r_{k_j}(2|\vec k|_{j-1}+\beta_j-\beta_0-1,\beta_{j+1}-\beta_j-1,|\vec k|_{j-1}-x_{j+1}-1, \\
 	& \qquad |\vec k|_{j-1}+\beta_j+x_{j+1},-|\vec k|_{j-1}+x_j) 
 \end{align*}
 with $\vec k:=(k_1,\ldots, k_{n-2}) \in \mathbb{N}^{n-2}$, $\vec x =(x_1, \ldots, x_{n-2})$, $\vec \beta:= (\beta_0, \ldots, \beta_{n-1})\in \mathbb{R}^{n-2}$ and $N:=x_{n-1} \in \mathbb{N}$. We also introduced the notation $|\vec k|_j=\sum_{i=1}^j k_i$ and assume that $| \vec k |_{n-2} \leq N$.
\end{definition}
In Proposition \ref{DiscreteRankOne} we introduced the operator $\mathcal{L}_1$. This operator has the univariate Racah polynomials as eigenvectors. In \cite{Geronimo&Iliev-2010} J. S. Geronimo en P. Iliev defined a set of operators $\{ \mathcal{L}_i \}$. They proved that this set has the multivariate Racah polynomials as common eigenvectors. We introduce them here:
\begin{definition}\label{ShiftOperator} (See also \cite{Geronimo&Iliev-2010})
The Racah operator is the shift operator:
\[ 
\mathcal{L}_{j}=\sum_{\substack{\vec \nu \in \{ -1,0,1\}^{j}\\ \vec \nu\neq \vec 0}} G_{\vec \nu}( E_{\vec \nu}-1).
\]
$E_{\vec \nu}$ is a shift operator defined as follows. Let $E_{x_i}^{\nu_i}(f(x_j))=f(x_j+\delta_{ij}\nu_i)$. Then we define $E_{\vec \nu}=E_{x_1}^{\nu_1}E_{x_2}^{\nu_2}\ldots E_{x_{j}}^{\nu_{j}}$.
The $G_{\vec \nu}$ are rational functions in the variables $x_0, x_1, \ldots , x_{j+1}$ and $\beta_0, \ldots, \beta_{j+1}$ and are defined as follows. We introduce the following functions:
\begin{align*}
B_i^{0,0}&:=x_i(x_i+\beta_i)+x_{i+1}(x_{i+1}+\beta_{i+1})+\frac{(\beta_i+1)(\beta_{i+1}-1)}{2} \\
B_i^{0,1}&:=(x_{i+1}+x_i+\beta_{i+1})(x_{i+1}-x_i+\beta_{i+1}-\beta_i) \\
B_i^{1,0}&:=(x_{i+1}-x_i)(x_{i+1}+x_{i}+\beta_{i+1}) \\
B_i^{1,1}&:=(x_{i+1}+x_i+\beta_{i+1})(x_{i+1}+x_i+\beta_{i+1}+1).
\end{align*}
Let $I_i f(x_i):=f(-x_i-\beta_i)$. We extend $B^{s,t}$ by defining:
\begin{align*}
B_i^{-1,t}&:=I_i(B_i^{1,t}) \\
B_i^{s,-1}&:=I_{i+1}(B_i^{s,1}) \\
B_i^{-1,-1}&:=I_i(I_{i+1}(B_i^{1,1})).
\end{align*}
We also introduce
\begin{align*}
b_i^0&:=(2x_i+\beta_i+1)(2x_i+\beta_i-1) \\
b_i^1&:=(2x_i+\beta_i+1)(2x_i+\beta_i) \\
b_i^{-1}&:=I_{i}(b_i^{1}).
\end{align*}
Let $|\vec \nu|_0$ be the number of zeroes appearing in $\vec \nu$. Then $G_{\vec \nu}$ is
\[
G_{\vec \nu}:=2^{|\vec \nu|_0}\frac{\prod_{i=0}^{j} B_i^{\nu_i,\nu_{i+1}}}{\prod_{i=1}^{j} b_i^{\nu_i}}.
\]
When we need to show explicitly the variables on which the Racah operator depends, we denote
\[
\mathcal{L}_j(x_0,x_1,\ldots,x_{j+1},\beta_0,\ldots,\beta_{j+1},E_{x_1},\ldots,E_{x_j}).
\]
\end{definition} 
The following proposition is Theorem 3.9. in \cite{Geronimo&Iliev-2010}
\begin{proposition}\label{EigenvaluesBispectral}
 For $j \in [p]$ the operators $\mathcal{L}_j$ act as follows on the multivariate Racah polynomials:
 \[
  \mathcal{L}_jR_{p}(\vec k; \vec x; \vec \beta ;N) =-|\vec k|_{j}(|\vec k|_{j}-1+\beta_j-\beta_0)R_{p}(\vec k; \vec x; \vec \beta ;N).
 \]
 The operators $\mathcal{L}_j$ commute with each other. 
\end{proposition}

\section{The main Theorem}\label{TheMainTheorem}
Let us summarize the previous sections. Proposition \ref{DiscreteRankOne} shows that there exists a discrete realization of $R(3)$ acting on the univariate Racah polynomials. It also mentions the existence of an operator $\mathcal{L}_1$ having the univariate Racah polynomials as eigenvectors. For both the univariate Racah polynomials and the operator $\mathcal{L}_1$ there exists multivariate generalizations. This seems to suggests that there exists a realization of the higher rank Racah algebra $R(n)$ in terms of difference operators encoding the multispectral properties of the multivariate Racah polynomials. This is the essence of the subsequent Theorem \ref{MainTheorem}. Before we formulate our main Theorem we define an additional isomorphism $\sigma$:
\begin{align*}
  \textup{Alg}[x_0,  \ldots, x_s;\beta_0, \ldots \beta_s; E_{1}, \ldots , E_{s}] &\rightarrow \textup{Alg}[x_1, \ldots, x_{s+1};\beta_1, \ldots \beta_{s+1}; E_{2}, \ldots , E_{s+1}] \\
 \sigma(x_i)&=x_{i+1} \\
 \sigma(\beta_i)&=\beta_{i+1}\\
 \sigma(E_{x_i})&=E_{x_{i+1}} .
 \end{align*}
By Lemma \ref{Lind} the set $\mathcal{B}_2=\{ C_{[i \ldots j]} | i \leq j\}$ is a generating set for $R(n)$. It thus suffices to define these operators in the discrete realization. 
 
 \begin{samepage}
\begin{theorem}\label{MainTheorem}
Using $\kappa(x,\beta)$ given in formula (\ref{kappa}) and $\mathcal{L}_i$ given in Definition \ref{ShiftOperator}, define the following operators:
\begin{align}
C_{[m]}&=\kappa(x_{m-1},\beta_{m-1}) \label{OperatorOne} & \text{ with } \quad 1\leq m  \leq n \\
C_{[2\ldots m+1]}&=-\mathcal{L}_{m-1}+\kappa(0,\beta_m-\beta_0-1) \label{OperatorTwo} & \text{ with } \quad 1\leq m  \leq n-1\\
C_{[p\ldots q]}&=\sigma^{p-2}(C_{[2\ldots q-p+2]}) \label{OperatorThree}, & \text{ with } \quad 2<p\leq q\leq n
\end{align}
and set $x_0=0$ and $x_{n-1}=N$.
The algebra generated by these operators is a discrete realization of $R(n)$.
\end{theorem}
\end{samepage}
The proof of this Theorem is given in Section \ref{DiscreteRealization}. The special case of the Hahn polynomials was previously considered in \cite{Iliev&Xu-2017-2}. The general case was considered in \cite{Iliev-2017}. In this article the same algebra was constructed but in a different way. More specifically, certain generators were given through a fourth order equation in the Racah operators. Theorem \ref{MainTheorem} on the other hand gives explicit expressions for all generators. We first give some examples to illustrate our result.

\begin{example} \textit{The rank one case.}
  We expect to retrieve Proposition \ref{DiscreteRankOne}. The central operators become, knowing that $\mathcal{L}_0=0$,
  \[ C_1=\kappa(0,\beta_0), \quad C_2= \kappa(0,\beta_1-\beta_0-1), \quad C_3=\sigma(C_2)=\kappa(0,\beta_2-\beta_1-1). \]
  and
  \[ C_{[3]}=\kappa(x_2,\beta_2). \]
  The operators $C_{12}$ and $C_{23}$ become
 \begin{align*}
 	C_{12}&=\kappa(x_1,\beta_1), \\
	C_{23}&=-\mathcal{L}_1(0,x_1,x_2,\beta_0,\beta_1,\beta_2,E_{x_1})+\frac{(\beta_2-\beta_0)(\beta_2-\beta_0-2)}{4}.
 \end{align*}
 This is exactly as in Proposition \ref{DiscreteRankOne} with $x=x_1$ and $N=x_2$. 
\end{example}

\begin{example} \textit{The rank two case.}  The algebra $R(4)$ is generated by the following elements:

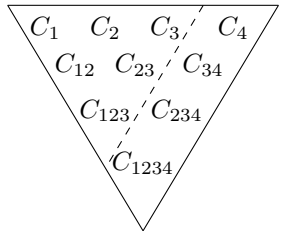
\begin{center}
\begin{tikzpicture}
	\draw (0,0)--(3.6,0);
	\draw (3.6,0)--(1.8,-3);
	\draw (1.8,-3)--(0,0);
	\draw[dashed] (2.6,0)--(1.3,-3*26/36);
	\draw (0.5,-0.3) node {$C_1$};
	\draw (1.3,-0.3) node {$C_2$};
	\draw (2.1,-0.3) node {$C_3$};
	\draw (3,-0.3) node {$C_4$};
	\draw (0.9,-0.8) node {$C_{12}$};
	\draw (1.7,-0.8) node {$C_{23}$};
	\draw (2.6,-0.8) node {$C_{34}$};
	\draw (1.3,-1.4) node {$C_{123}$};
	\draw (2.25,-1.4) node {$C_{234}$};
	\draw (1.8,-2.1) node {$C_{1234}$};
\end{tikzpicture}
\end{center}
The new elements are on the right-hand side of the dashed line. By Theorem \ref{MainTheorem} we have
\[ C_4:=\kappa(0,\beta_3-\beta_2-1), \quad C_{[4]}=\kappa(x_3,\beta_3), \quad C_{234}=-\mathcal{L}_2+\frac{(\beta_3-\beta_0)(\beta_3-\beta_0-2)}{4}.
\]
The following operator is the interesting one:
\[
C_{34}=\sigma(C_{23})=-\mathcal{L}_1(x_1,x_2,x_3,\beta_1,\beta_2,\beta_3,E_{x_2})+\frac{(\beta_3-\beta_1)(\beta_3-\beta_1-2)}{4}.
\]
Notice that this is a three-point difference operator in the $x_2$ direction compared to $C_{23}$ being a three-point difference operator in the $x_1$ direction. One should note that these operators coincide, up to a gauge factor, with the operators found by S. Post in \cite{Post}. This article also shows that the $R(3)$ sits a number of times inside $R(4)$. This is akin to Proposition \ref{Racah1}.
\end{example}

\section{Multivariate Racah polynomials as connection coefficients}\label{MultivariateRacahPolynomialsAsConnectionCoefficients}

\subsection{The Dunkl model}
Our starting point is the Dunkl model constructed in \cite{DeBie&Genest$vandeVijver&Vinet}. We define the Dunkl-operator as a deformation of the partial derivative $\partial_{y}$:
\begin{align*}
T=\partial_{y}+\frac{\mu}{y}(1-r)
\end{align*}
where the real number $\mu> 0$ is the deformation parameter, $\partial_{y}$ is the partial derivative with respect to $y$ and $r$ is the reflection operator i.e. $rf(y)=f(-y)$.
We realize the algebra $\mathfrak{su}(1,1)$ by introducing the following operators into formula (\ref{su11}).
\[ 
	J_+=\frac{y^2}{2}, \qquad J_-=\frac{T^2}{2} , \qquad A_0=\frac{1}{2}(\mathbb{E}+\gamma)
\]
with $\mathbb{E}=y\partial_y$ and $\gamma=\frac{1}{2}+\mu$. This way we construct a realization of $R(n)$ using Dunkl-operators. Moreover, in this realization the Racah algebra becomes the symmetry algebra of the $\mathbb{Z}_2^n$ Dunkl-Laplacian. The construction is as follows. Let
\begin{equation*}
\begin{aligned}
 T_j&=\underbrace{1\otimes \ldots \otimes1}_{j-1\text{ times}} \otimes T \otimes \underbrace{1\otimes \ldots 	\otimes1}_{n-j \text{ times}} \\
 	&=\partial_{y_j}+\frac{\mu_j}{y_j}(1-r_j).
\end{aligned}
 \end{equation*}
 This operator $T_j$ only acts on the variable $y_j$. The set of operators $T_j$ is the set of $\mathbb{Z}_2^n$ Dunkl-operators.
The $\mathbb{Z}_2^n$ Dunkl-Laplacian becomes
\[
 \Delta=\sum_{j=1}^n T_j^2.
\]
This operator has $n$ deformation parameters $\mu_j$, $j \in [n]$. As the Racah algebra is a symmetry algebra, it preserves $\ker \Delta$, the space of Dunkl-harmonics. Let $\mathcal{H}_a=\ker \Delta \cap \mathcal{P}_a$ be the space of Dunkl-harmonics which are also polynomials of degree $a$. This space $\mathcal{H}_a$ is not an irreducible representation but can be written as a direct sum of irreducible subspaces. 
\[
\mathcal{H}_a=\bigoplus_{\vec \epsilon \in \{0,1\}^n} \mathcal{H}_a^{\vec \epsilon}\quad \text{ with }  \quad r_i\mathcal{H}_a^{\vec \epsilon}=(-1)^{\epsilon_i}\mathcal{H}_a^{\vec \epsilon}.
\]
The space $\mathcal{H}_a^{\vec \epsilon}$ is the space of Dunkl-harmonics of degree $k$ and of odd or even degree in the variable $y_i$ depending on the value of $\epsilon_i$.

Bases for these spaces were constructed such that they diagonalize a labelling Abelian algebra  (see Definition \ref{labelling}). For example, consider the chain $\mathcal{F}: [2] \subset \left[3\right] \subset \ldots \subset [n-1]$. There exists a basis of Dunkl-harmonics diagonalized by the labelling Abelian algebra $\mathscr{Y}_{\mathcal{F}}$ which we will denote by $\{ \psi_{\vec \ell } \}$ with quantum numbers $\vec \ell=(\ell_1,\ldots \ell_{n-2}) \in \mathbb{N}^{n-2}$. According to Proposition $8$ in \cite{DeBie&Genest$vandeVijver&Vinet} we have:
\[
 C_{[m]}\psi_{\vec \ell}=\frac{1}{4}\left(d_m+\gamma_{[m]}\right)\left(d_m+\gamma_{[m]}-2\right)\psi_{\vec \ell}
 \]
with $d_m:=\epsilon_m+\sum_{j=1}^{m-1}( 2\ell_j+\epsilon_j)$. The constants $\gamma_{[m]}$ depend on the deformation parameters and equal 
\[
	\gamma_{[m]}=\frac{m}{2}+\sum_{j=1}^{m} \mu_j.
\]
To make the dependence on the quantum numbers $\vec \ell$ more clear, we introduce the following notation. Let $\beta_{m-1}=\gamma_{[m]}-1+\sum_{j=1}^m\epsilon_j$
and $|\vec \ell |_m=\sum_{j=1}^m \ell_j$. We get
\begin{equation*}
\begin{aligned} 
C_{[m]}\psi_{\vec \ell}&=\left(|\vec \ell |_{m-1}+\frac{\beta_{m-1}+1}{2}\right)\left(|\vec \ell |_{m-1}+\frac{\beta_{m-1}-1}{2}\right)\psi_{\vec \ell} \\
					&=\kappa(|\vec \ell |_{m-1},\beta_{m-1})\psi_{\vec \ell} .
\end{aligned}
\end{equation*}

Consider a second chain : $\mathcal{G}: [2\ldots3] \subset [2 \ldots4] \subset \ldots \subset [2 \ldots n]$. The basis diagonalised by $\mathscr{Y}_{\mathcal{G}}$ is denoted by $\{\varphi_{\vec k}\}$.
We have the following action in the Dunkl model:
\[
C_{[2 \ldots m+1]}\varphi_{\vec k}=\frac{1}{4}\left(d'_m+\gamma_{[2\ldots m]}\right)\left(d'_m+\gamma_{[2\ldots m]}-2\right)\varphi_{\vec k}
\]
with $d'_m:=\epsilon_{m+1}+\sum_{j=1}^{m-1} 2k_j+\sum_{j=2}^{m} \epsilon_j$ and $\epsilon_{n+1}:=\epsilon_{1}$. Using the notation above we get:
\begin{equation}\label{EigenvalueTwo}
\begin{aligned}
C_{[2 \ldots m+1]}\varphi_{\vec k}&=\left(|\vec k |_{m-1}+\frac{\beta_{m}-\beta_0}{2}\right)\left(|\vec k |_{m-1}+\frac{\beta_{m}-\beta_0-2}{2}\right)\varphi_{\vec k}\\
&=\kappa(|\vec k|_{m-1},\beta_{m}-\beta_0-1)\varphi_{\vec k}.
\end{aligned}
\end{equation}
In general the spectrum of the operator $C_{[p \ldots q ]}$ is given in the following lemma:

\begin{lemma}\label{spectra}
The operator $C_{[p \ldots q]}$ is diagonalizable in $\mathcal{H}_a^{\vec \epsilon}$.
 Let $d^*=\sum_{j=1}^{q-p} 2m_j+\sum_{j=p}^{q} \epsilon_j$. Let $  N:= \frac{a -\sum_{j=1}^n \epsilon_j }{2}$.
 The spectrum of $C_{[p \ldots q]}$ in $\mathcal{H}_a^{\vec \epsilon}$ is given by
 \[
  \Big\{ \frac{1}{4}\left(d^*+\gamma_{[p\ldots q]}\right)\left(d^*+\gamma_{[p \dots q]}-2\right) \Big |  |\vec m|_{p-q} \leq N\Big\}
  \]
  or equivalently
  \[
   \Big\{ \kappa(|\vec m|_{q-p},\beta_{q-1}-\beta_{p-2}-1) \Big |  |\vec m|_{p-q} \leq N\Big\}
  \]
  with $\beta_{-1}:=-1$.
\end{lemma}
\begin{proof}
 We give a sketch of the proof. In Section 5.2 of \cite{DeBie&Genest$vandeVijver&Vinet} a basis diagonalizing a labelling Abelian algebra is constructed using the CK-extension and Fischer decomposition. Using this method one constructs a basis by adding one variable at a time. The order in which the variables are added, matters. If we choose to add variables $y_p, y_{p+1}, \ldots, y_{q}$ first, and then the remaining variables in any order, one obtains a basis diagonalizing $C_{[p \ldots q]}$ and the lemma follows.
\end{proof}

We also give a sketch of the proof of irreducibility of $\mathcal{H}_a^\epsilon$. In \cite{Iliev-2018} one obtains the irreducibility for a different representation of $R(n)$.
\begin{lemma}\label{irreducibility}
The space of Dunkl-harmonics $\mathcal{H}_a^{\vec \epsilon}$ is an irreducible representation for the realization of $R(n)$ in Dunkl operators. 
\end{lemma}
\begin{proof}
Consider the basis $\{ \psi_{\vec \ell}\}$ diagonalizing $\mathscr{Y}_\mathcal{F}$.  When acting with each operator of $\mathscr{Y}_\mathcal{F}$ on a basis element $\psi_{\vec \ell}$ on obtains a set of eigenvalues by Lemma \ref{spectra}
\[
	\big\{ \kappa(|\vec \ell|_{1},\beta_{1}),\kappa(|\vec \ell|_{2},\beta_{2}),\ldots,\kappa(|\vec \ell|_{n-2},\beta_{n-2}) \big\}.
\]
This set is unique for every $\psi_{\vec \ell}$ so the representation is non-degenerate.
Consider now the operator $C_{\{j+1,j+2\}}$. It commutes with all generators in $\mathscr{Y}_\mathcal{F}$ except $C_{[j+1]}$ by Lemma \ref{Commute}. This means, if we take
\[
C_{\{ j+1,j+2\}}\psi_{\vec \ell}=\sum_{\vec \ell'} b_{\vec \ell'}\psi_{\vec \ell'},
\]
that $b_{\vec \ell'} \neq 0$ only if $|\vec \ell|_p=|\vec \ell'|_p$ for $p \in [n] \backslash \{ j\}$. So if one considers $C_{\{ j+1,j+2\}}$ to be an operator acting on the vectors $(|\vec \ell|_{1},|\vec \ell|_{2}\ldots |\vec \ell|_{n-2} )$, it will only change the quantum number $|\vec \ell|_{j}$. When considering the rank one Racah algebra by Lemma \ref{Racah1}:
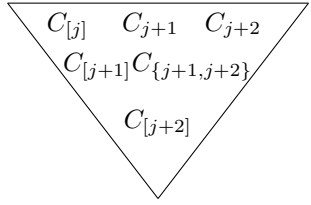
\begin{center}
\begin{tikzpicture}
\draw (-0.2,0)--(3.8,0);
\draw (3.8,0)--(1.8,-2.6);
\draw (1.8,-2.6)--(-0.2,0);
\draw (0.6,0) node[below] {$C_{[j]}$};
\draw (1.7,0) node[below] {$C_{j+1}$};
\draw (2.8,0) node[below] {$C_{j+2}$};

\draw (1.8,-0.85) node {$C_{[j+1]}  C_{\{ j+1,j+2\}}$};
\draw (1.8, -1.6) node { $C_{[j+2]}$ };
\end{tikzpicture}
\end{center}
one observes that this rank one Racah algebra commutes with all generators in $\mathscr{Y}_\mathcal{F}$ except $C_{[j+1]}$ again by Lemma \ref{Commute}. It will therefore only change the quantum number $|\vec \ell|_{j}$. Repeating the analysis done for the rank one Racah algebra in \cite{Genest&Vinet&Zhedanov-2014-3}, one concludes that $C_{\{ j+1,j+2\}}$ is a tridiagonal operator for the quantum number  $|\vec \ell|_{j}$. In this model, the Dunkl model, one can build ladder operators acting on the quantum number $|\vec \ell|_{j}$ using these tridiagonal operators similarly to Lemma 8 in \cite{DeBie&DeClercq&vandeVijver}. Subsequently one can build, in this vein, starting from a single basis element every other basis element in $\{ \psi_{\vec \ell}\}$ proving the irreducibility.
\end{proof} 

Consider again the bases $\{ \psi_{\vec \ell}\}$ diagonalizing $\mathscr{Y}_\mathcal{F}$ and $\{ \varphi_{\vec k} \}$ diagonalizing $\mathscr{Y}_\mathcal{G}$. The connection coefficients between these two bases are given by the functions $R_{\vec k}$ defined as:
\[
\bra{\varphi_{\vec k}}\ket{\psi_{ \vec \ell}}=:R_{\vec k}(\vec \ell) .
\]
We denoted the connection coefficients using the bra-ket notation. 
It is possible to lift the action of $R(n)$ to the connection coefficients in the following way:
\begin{equation}\label{Lift1}
C_A R_{\vec k}(\vec \ell)= C_A\bra{\varphi_{\vec k}}\ket{\psi_{ \vec \ell}}:=\bra{\varphi_{\vec k}}C_A\ket{\psi_{ \vec \ell}} .
\end{equation}
This definition preserves the action and relations of the higher rank Racah algebra. We want to find expressions for the operators $C_A$ acting as shift operators on the grid formed by $\vec \ell$. As an example consider the operators of $\mathscr{Y}_{\mathcal{F}}$:
\begin{equation}\label{FirstOperator}
C_{[m]}R_{\vec k}(\vec \ell)=\bra{\varphi_{\vec k}}C_{[m]}\ket{\psi_{ \vec \ell}}=\kappa(|\vec \ell |_{m-1},\beta_{m-1})R_{\vec k}(\vec \ell).
\end{equation}
Introduce the notation $x_m:=|\vec \ell |_{m}$. The discrete realization becomes:
\begin{equation}\label{TheoremLineOne}
  C_{[m]}=\kappa(x_{m-1},\beta_{m-1}).
\end{equation}
This is the first set of operators presented in formula (\ref{OperatorOne}) of Theorem \ref{MainTheorem} . To find the discrete realization of the other operators we need two steps. The first step is proving the following Theorem:

\begin{theorem} \label{MultivariateRacah}
Let $\{\psi_{\vec \ell}\}$ be a basis of an irreducible space of Dunkl-harmonics diagonalized by the labelling Abelian algebra $ \mathscr{Y}^{\psi}=\langle C_{[t]}\rangle_{2 \leq t \leq n-1}$  and  $\{ \varphi_{\vec k}\}$ a basis diagonalized by $\mathscr{Y}^{\varphi}= \langle C_{[2\ldots t]}\rangle_{3 \leq t \leq n}$.   The connection coefficients between these bases  $R_{\vec k}(\vec \ell)=\bra{\varphi_{\vec k}}\ket{\psi_{\vec \ell}}$ are up to a gauge factor equal to the multivariate Racah polynomials $R_{n-2}(\vec k; \vec x; \vec \beta ;N)$ with $x_m:=|\vec \ell |_{m}$.
 \end{theorem}
 When we have proven that the connection coefficients between these bases are indeed the multivariate Racah polynomials, we will show that the action of $C_A$ coincides with the action of a Racah operator. This will prove Theorem \ref{MainTheorem}. We also dedicate a section to the gauge factors as this depends on how the chosen bases have been normalized.

\subsection{ Proof of Theorem \ref{MultivariateRacah}}
\begin{proof}
We break down the construction of the connection coefficients in small steps using intermediate bases diagonalised by labelling Abelian algebras. The starting basis and algebra are $\mathscr{B}_0:=\{ \psi_{\vec \ell}\}$ and $\mathscr{Y}_0:=\mathscr{Y}^\psi=\langle C_{[m]} \rangle$. The finishing basis and algebra are $\mathscr{B}_{n-2} :=\{ \varphi_{ \vec k}\}$ and  $\mathscr{Y}_{n-2}:=\mathscr{Y}^\varphi=\langle C_{[2 \ldots m]} \rangle$. The intermediate labelling Abelian algebras are constructed in such a way that two subsequent algebras differ by only one generator:
\[ 
\mathscr{Y}_{j}:=\langle C_{23},C_{234}, \ldots, C_{[2\ldots j+2]}, C_{[j+2]}, \ldots C_{[n-1]}\rangle 
\]
and we associate to this algebra the diagonalised basis $\mathscr{B}_j$.  The following diagram shows the different intermediate steps. In the upper row, one finds the intermediate bases $\mathcal{B}_j$. Underneath each basis $\mathcal{B}_j$ the generators of each related labelling Abelian algebra $\mathscr{Y}_j$ are provided. The generators by which two subsequent labelling Abelian algebras differ, are indicated by a box.

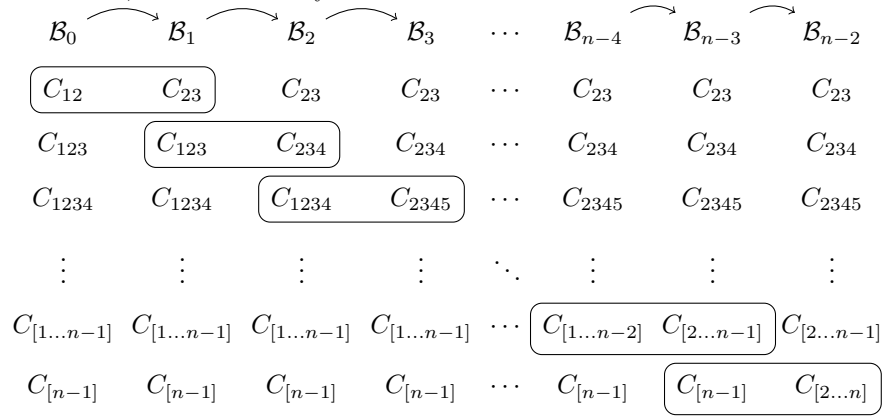
\begin{tikzpicture}[->]
 \matrix(M)[matrix of nodes, column sep=0 cm, row sep=0.2cm]{
  $\mathcal{B}_0$ & $\mathcal{B}_1$ & $\mathcal{B}_2$	& $\mathcal{B}_3$		& $\cdots$ & $\mathcal{B}_{n-4}$	& $\mathcal{B}_{n-3}$ 	& $\mathcal{B}_{n-2}$	\\
  $C_{12}$ 	& $C_{23}$	& $C_{23}$ 			& $C_{23}$			& $\cdots$ & $C_{23}$			& $C_{23}$			& $C_{23}$			\\
  $C_{123}$ 	& $C_{123}$	& $C_{234}$			& $C_{234}$			& $\cdots$ & $C_{234}$			& $C_{234}$			& $C_{234}$			\\
  $C_{1234}$ 	& $C_{1234}$	& $C_{1234}$			& $C_{2345}$			& $\cdots$ & $C_{2345}$			& $C_{2345}$			& $C_{2345}$			\\
  \vdots 		& \vdots 		& \vdots				& \vdots				& $\ddots$ & $\vdots$			& $\vdots$			& \vdots				\\
  $C_{[1\ldots n-1]}$	&$C_{[1\ldots n-1]}$	& $C_{[1\ldots n-1]}$		& $C_{[1\ldots n-1]}$		& $\cdots$ & $C_{[1\ldots n-2]}$		& $C_{[2\ldots n-1]}$		& $C_{[2\ldots n-1]}$		\\
  $C_{[n-1]}$	& $C_{[n-1]}$	& $C_{[n-1]}$			& $C_{[n-1]}$			& $\cdots$ & $C_{[n-1]}$			& $C_{[n-1]}$			& $C_{[2\ldots n]}$			\\
  };
 \path (M-1-1) edge[bend left] (M-1-2); 
 \path (M-1-2) edge[bend left] (M-1-3);
 \path (M-1-3) edge[bend left] (M-1-4);
 \path (M-1-6) edge[bend left] (M-1-7);
 \path (M-1-7) edge[bend left] (M-1-8);
 \node[draw=black, rounded corners=1ex, fit=(M-2-1)(M-2-2),inner sep=1pt]{};
 \node[draw=black, rounded corners=1ex, fit=(M-3-2)(M-3-3),inner sep=1pt]{};
 \node[draw=black, rounded corners=1ex, fit=(M-4-3)(M-4-4),inner sep=1pt]{};
 \node[draw=black, rounded corners=1ex, fit=(M-6-6)(M-6-7),inner sep=1pt]{};
 \node[draw=black, rounded corners=1ex, fit=(M-7-7)(M-7-8),inner sep=1pt]{};
 \end{tikzpicture}

Let us perform the first step. We want to find the connection coefficients between 
$\mathscr{B}_0$ and $\mathscr{B}_1$. Instead of considering the full space the Racah algebra acts on, consider the eigenspaces of $\mathscr{A}_1:=\mathscr{Y}_0\cap\mathscr{Y}_1$. This set $\mathscr{A}_1$ is generated by the operators $\{ C_{[m]} \}_{m \geq 3} $. The algebra $R^{1,2,3}(3):=\langle C_1,C_2,C_3,C_{12},C_{23},C_{123}\rangle$ commutes with $\mathscr{A}_1$ and therefore preserves its eigenspaces. Observe that $R^{1,2,3}(3)$ is a rank one Racah algebra:
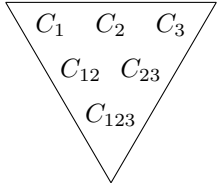
\begin{center}
 \RacahOne{C_1}{C_2}{C_3}{C_{12}}{C_{23}}{C_{123}}
 \end{center}
The operators $C_{12}$ and $C_{23}$ sitting in this algebra, diagonalize $\mathscr{B}_0$ and $\mathscr{B}_1$ respectively. According to Proposition \ref{UnivariateRacah} the connection coefficients are up to a gauge constant equal to 
\[ 
 r_{k_1}(\beta_1-\beta_0-1,\beta_{2}-\beta_1-1,-x_{2}-1, \beta_1+x_{2};x_1)
 \]
 by formula (\ref{ConnRacah}). We derive the constants $\beta_0$, $\beta_1$, $\beta_2$, $x_1$, $x_2$ and $k_1$ from the eigenvalues of the generators of $R^{1,2,3}(3)$ by Lemma \ref{spectra}:
\begin{align*}
   C_1&:\kappa(0,\beta_0) \\
   C_2&:\kappa(0,\beta_1-\beta_0-1)\\
   C_3&:\kappa(0,\beta_2-\beta_1-1) \\
   C_{12}&: \kappa(x_1,\beta_1) \\
   C_{23}&:\kappa(k_1,\beta_2-\beta_0-1) \\
   C_{123}&:\kappa(x_2,\beta_2).
\end{align*}

Using this first step we construct the connection coefficients between any two successive steps. Consider $\mathscr{A}_j:=\mathscr{Y}_{j-1}\cap \mathscr{Y}_{j}$. Then $\mathscr{A}_j$ is generated by $\{ C_{[m]} \}_{m>j+1} \cup \{ C_{[2\ldots m]} \}_{m\leq j+1}$. Consider the eigenspaces of $\mathscr{A}_j$. The rank one Racah algebra $R^{1,[2\ldots j+1],j+2}(3):=\langle C_1, C_{[2\ldots j+1]}, C_{j+2}, C_{[j+1]}, C_{[2\ldots j+2]} ,C_{[j+2]}\rangle$ commutes with $\mathscr{A}_j$. We represent this graphically as
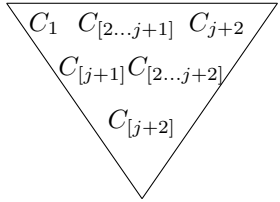
\begin{center}
\begin{tikzpicture}
\draw (0,0)--(3.6,0);
\draw (3.6,0)--(1.8,-2.6);
\draw (1.8,-2.6)--(0,0);
\draw (0.5,0) node[below] {$C_1$};
\draw (1.6,0) node[below] {$C_{[2\ldots j+1]}$};
\draw (2.8,0) node[below] {$C_{j+2}$};

\draw (1.8,-0.9) node {$C_{[ j+1]}  C_{[2\ldots j+2]}$};
\draw (1.8, -1.6) node { $C_{[j+2]}$ };
\end{tikzpicture}
\end{center}

This algebra therefore preserves these eigenspaces of $\mathscr{A}_j$. The operators $C_{[j+1]}$ and $C_{[2\ldots j+2]}$ sitting in this algebra, diagonalize $\mathscr{B}_{j-1}$ and $\mathscr{B}_j$ respectively. According to Proposition \ref{UnivariateRacah} the connection coefficients must therefore be, up to a gauge factor, the Racah polynomials by formula (\ref{ConnRacah})
\[  
r_{k'_1}(\beta'_1-\beta'_0-1,\beta'_{2}-\beta'_1-1,-x'_{2}-1, \beta'_1+x'_{2};x'_1).
\]
We find the variables $\beta'_0$, $\beta'_1$, $\beta'_2$, $x'_1$, $x'_2$ and $k'_1$ by comparing the eigenvalues of the generators of $R^{1,[2\ldots j+1],j+2}(3)$, see Lemma \ref{spectra}, with the eigenvalues of the generators of $R^{1,2,3}(3)$ given in formula (\ref{spectra1}).
\begin{align*}
   C_1&:                         &      \kappa(0,\beta_0) &=\kappa(0,\beta'_0)  \\
   C_{[2\ldots j+1]}&:      &      \kappa(|\vec k|_{j-1},\beta_j-\beta_0-1) &=\kappa(0,\beta'_1-\beta'_0-1) \\
   C_{j+2}&:                    &      \kappa(0,\beta_{j+1}-\beta_j-1)&=\kappa(0,\beta'_2-\beta'_1-1) \\
   C_{[j+1]}&:                  &      \kappa(x_j,\beta_j)&= \kappa(x'_1,\beta'_1) \\
   C_{[2\ldots j+2]}&:       &      \kappa(|\vec k|_{j},\beta_{j+1}-\beta_0-1)&=\kappa(k'_1,\beta'_2-\beta'_0-1) \\
   C_{[j+2]}&:                   &      \kappa(x_{j+1},\beta_{j+2})&=\kappa(x'_2,\beta'_2).
\end{align*}

Solving this set of equations we find:
\begin{equation}\label{ConnSubs}
\begin{aligned}
\beta'_0&=\beta_0 &  \beta'_1&=2| \vec k|_{j-1}+\beta_j &  \beta'_2&=2| \vec k|_{j-1}+\beta_{j+1} \\
x'_1&=x_j-| \vec k|_{j-1} &   x'_2&=x_{j+1}-| \vec k|_{j-1} & k'_1&=k_j.
\end{aligned}
\end{equation}

Hence, the connection coefficients become 
\begin{equation*}
r_{k_j}(2|\vec k|_{j-1}+\beta_j-\beta_0-1,\beta_{j+1}-\beta_j-1,|\vec k|_{j-1}-x_{j+1}-1,|\vec k|_{j-1}+\beta_j+x_{j+1},-|\vec k|_{j-1}+x_j) 
\end{equation*}
which are exactly the factors of the multivariate Racah polynomials $R_{n-2}(\vec k; \vec x; \vec \beta ;N)$ as given in Definition \ref{defRacah}.
\end{proof}
In \cite{Iliev&Xu-2017} connection coefficients between bases in terms of Jacobi polynomials were obtained using a different direct method. These connection coefficients turned out to be multivariate Racah polynomials. This is a result similar to the theorem proven here as the bases of Dunkl-harmonics are also expressed in terms of Jacobi polynomials (See Section 5.3 in \cite{DeBie&Genest$vandeVijver&Vinet}).

\subsection{Gauge coefficients}
If we assume the bases $\psi_{\vec \ell }$ and $\varphi_{\vec k}$ to be normalized, the connection coefficients are normalized too up to a phase constant. They satisfy
\[
	\sum_{\vec k \in V_k} R_{\vec \ell'}( \vec k)R_{\vec k}(\vec \ell)=\delta_{\vec \ell, \vec \ell'}
\]
with 
\[
 V_k=\{ \vec k \in \mathbb{N}_0^{n-2}   |\, |\vec k|_{n-2} \leq N \}.
\]
This enables us to write down the gauge coefficients explicitly. We compare the orthogonality relation of the connection coefficients with the orthogonality relation of the multivariate Racah polynomials. The orthogonality relation of the multivariate Racah polynomials can be found by writing the multivariate Racah polynomials as a product of univariate Racah polynomials as we have done in the proof of Theorem \ref{MultivariateRacah}. We normalize each factor considering the orthogonality relation (\ref{ConnOrtho}). In this relation substitute as in formula (\ref{ConnSubs}). This leads to the following: 
\begin{align*}
  \rho_j=\frac{\Gamma(\beta_j+x_j+| \vec k|_{j-1})\Gamma(\beta_j-\beta_0+x_j+| \vec k|_{j-1})}{\Gamma(\beta_0+1+x_j-| \vec k|_{j-1})(x_j-| \vec k|_{j-1})!}(\beta_j+2x_j)  \\\times\frac{\Gamma(\beta_{j+1}+x_{j+1}+x_{j})\Gamma(\beta_{j+1}-\beta_j+x_{j+1}-x_j)}{\Gamma(\beta_j+x_{j+1}+1+x_j)(x_{j+1}-x_j)!}
\end{align*}
and
\begin{align*} 
\frac{1}{\lambda_{j}}:=\frac{\Gamma(2| \vec k|_{j-1}+k_j+\beta_{j+1}-\beta_0-1)}{\Gamma(k_j+\beta_{j+1}-\beta_j)\Gamma(| 2\vec k|_{j-1}+n_j+\beta_j-\beta_0)n_j!}(2| \vec k|_{j}+\beta_{j+1}-\beta_0-1)\\
\times\frac{\Gamma(\beta_0+x_{j+1}+1-| \vec k|_{j})(x_{j+1}-| \vec k|_{j})!}{\Gamma(| \vec k|_{j-1}+\beta_{j+1}-\beta_0+x_{j+1})\Gamma(\beta_{j+1}+x_{j+1}+| \vec k|_{j})} .
\end{align*}
By multiplying these weights $\rho_j$ and factors $\lambda_{j}$ we obtain the following weight for the multivariate Racah polynomials:
\[
\frac{\prod_{j=1}^p \rho_j}{\prod_{j=1}^p \lambda_j}=\omega_p(\vec x)\mu_p(\vec k)
\]
with
\[
\omega_p(\vec x)= \prod_{j=0}^p \frac{\Gamma(\beta_{j+1}+x_{j+1}+x_{j})\Gamma(\beta_{j+1}-\beta_j+x_{j+1}-x_j)}{\Gamma(\beta_j+x_{j+1}+1+x_j)(x_{j+1}-x_j)!} \prod_{j=1}^p (\beta_j+2x_j) 
\]
and
\begin{align*}
 \mu_p(\vec k)= \prod_{j=1}^p \frac{\Gamma(2| \vec k|_{j-1}+k_j+\beta_{j+1}-\beta_0-1)(2| \vec k|_{j}+\beta_{j+1}-\beta_0-1)}{\Gamma(k_j+\beta_{j+1}-\beta_j)\Gamma(| 2\vec k|_{j-1}+k_j+\beta_j-\beta_0)k_j!} \times \\
 \frac{\Gamma(\beta_0+N+1-| \vec k|_{p})(N-| \vec k|_{p})!}{\Gamma(\beta_{p+1}-\beta_0+N+| \vec k|_{p})\Gamma(\beta_{p+1}+N+| \vec k|_{p})}.
 \end{align*} 
 The expression for $\omega_p(\vec x)$ is similar to formula (3.8) in \cite{Geronimo&Iliev-2010}.
These formulas lead to
 \[
  \sum_{\vec x \in V_x} \omega_{n-2}(\vec x)\mu_{n-2}(\vec k) R_{n-2}(\vec k; \vec x ; \vec \beta;N)R_{n-2}(\vec k'; \vec x ; \vec \beta;N)=\delta_{\vec k,\vec k'}
 \]
 with
 \[ 
 V_x=\{ \vec x \in \mathbb{N}_0^{n-2}  | 0\leq x_1 \leq x_2 \leq \ldots \leq x_{n-2} \leq N \}.
 \]
 Notice that $V_x=V_k$.
The connection coefficients become
\begin{equation}
 R_{\vec k}(\vec \ell)=\sqrt{\omega_{n-2}(\vec x)\mu_{n-2}(\vec k)}R_{n-2}(\vec k; \vec x ; \vec \beta;N)
\end{equation}
with gauge coefficient  $\sqrt{\omega_{n-2}(\vec x)\mu_{n-2}(\vec k)}$ and $x_i:=|\vec l|_i$.

\section{A discrete realization}\label{DiscreteRealization}
As explained before in formula (\ref{Lift1}) we lift the action of the higher rank Racah algebra $R(n)$ from the Dunkl model to the connection coefficients by the following:
\[
   C_AR_{\vec k}(\vec \ell)=C_A\bra{\varphi_{\vec k}}\ket{\psi_{ \vec \ell}}:=\bra{\varphi_{\vec k}}C_A\ket{\psi_{ \vec \ell}} .
\]
From this point on we assume that the bases of Dunkl-harmonics $\psi_{ \vec \ell}$ and $\varphi_{ \vec k}$ are normalized so we know the explicit expression of the connection coefficients. Similarly to formula (\ref{FirstOperator}) we can now find the discrete realization of $C_{[2\ldots m+1]}$. We know the action of $C_{[2\ldots m+1]}$ on $\varphi_{ \vec k}$ by formula (\ref{EigenvalueTwo}).
\begin{align*}
 C_{[2\ldots m+1]}R_{\vec k}(\vec \ell)&=\bra{\varphi_{\vec k}}C_{[2\ldots m]}\ket{\psi_{ \vec \ell}} \\
 						&=\kappa(|\vec k|_{m-1},\beta_{m}-\beta_0-1)\bra{\varphi_{\vec k}}\ket{\psi_{ \vec \ell}}  \\
						&=\kappa(|\vec k|_{m-1},\beta_{m}-\beta_0-1)R_{\vec k}(\vec \ell).
\end{align*}
We need an operator with eigenvalues $\kappa(|\vec k|_{m-1},\beta_{m}-\beta_0-1)$ and eigenvectors the multivariate Racah polynomials which only acts on the $\vec x$ variable or equivalently acting only on $\vec \ell$ with $x_m=|\vec \ell|_{m}$. The Racah operator $\mathcal{L}_{m-1}$ is a good candidate. We adjust for the eigenvalues given in Proposition \ref{EigenvaluesBispectral}. This yields
\begin{align*}
\big(-\mathcal{L}_{m-1}(\vec x,\vec \beta,N)+\kappa(0,\beta_m-\beta_0-1)\big)R_{n-2}(\vec k; \vec x ; \vec \beta;N)\\
   =	\kappa(|\vec k|_{m-1},\beta_{m}-\beta_0-1)R_{n-2}(\vec k; \vec x ; \vec \beta;N).
\end{align*}
Taking into account the gauge factors the operator $C_{2 \ldots m+1}$ becomes:
\begin{equation}\label{GaugedTwo}
 C_{[2\ldots m+1]}=\sqrt{\omega_{n-2}(\vec x)}\big(-\mathcal{L}_{m-1}(\vec x,\vec \beta,N)+\kappa(0,\beta_m-\beta_0-1)\big)\sqrt{\omega_{n-2}(\vec x)}^{-1}.
\end{equation}

We return to the general case. Consider a general operator $C_A$. Let $ \ket{ \phi_{\vec m}}$ be a normalized basis such that $C_A$ acts diagonally on this basis. We let this operator act on $R_{\vec k}(\vec \ell)=\bra{ \psi_{\vec k}}\ket{\varphi_{\vec \ell}}$. This yields
\begin{equation}\label{MainTrick}
\begin{aligned} 
 C_A\bra{ \varphi_{\vec k}}\ket{\psi_{\vec \ell}}&=\sum_{\vec m} \bra{ \varphi_{\vec k}}\ket{ \phi_{\vec m}}   \bra{ \phi_{\vec m}}C_A\ket{ \psi_{\vec \ell}} \\
&=\sum_{\vec m} \bra{ \varphi_{\vec k}}\ket{ \phi_{\vec m}} \nu(\vec m)\bra{ \phi_{\vec m}}\ket{ \psi_{\vec \ell}} \\
&=\sum_{\vec m} \bra{ \varphi_{\vec k}}\ket{ \phi_{\vec m}} T_{\vec \ell}^{\phi}\bra{ \phi_{\vec m}}\ket{ \psi_{\vec \ell}} \\
&=T_{\vec \ell}^\phi \sum_{\vec m} \bra{ \varphi_{\vec k}}\ket{ \phi_{\vec m}} \bra{ \phi_{\vec m}}\ket{ \psi_{\vec \ell}} \\
&=T_{\vec \ell}^\phi \bra{ \varphi_{\vec k}}\ket{\psi_{\vec \ell}}.
\end{aligned}
\end{equation}
Here we introduced the operator $T_{\vec \ell}^\phi$. We will show its existence and will give its explicit expression later on. For now we have the following three assumptions.  First, this is an operator acting on functions on the grid $\vec \ell$. Second, it has  $\bra{ \phi_{\vec m}}\ket{ \psi_{\vec \ell}}$ as eigenvectors and eigenvalues $\nu(\vec m)$. Third and most importantly to note, we assume that $T_{\vec \ell}^\phi$ is independent of the grid $\vec m$. This means that we can pull the operator $T_{\vec \ell}^\phi$ out of the sum. In the remainder of this section we will choose $\ket{\phi_{\vec m}} $ so that $\bra{ \phi_{\vec m}}\ket{ \psi_{\vec \ell}}$ is a multivariate Racah polynomial. This way the $T_{\vec \ell}^\phi$ are the Racah operators $\mathcal{L}_i$ up to a constant and the three assumptions will be fullfilled.

\subsection{The operator $C_{34}$}
Before tackling the problem of finding the discrete realization of $R(n)$, we restrict ourselves to the easiest non-trivial case: the rank two case $R(4)$. The generators of the algebra are illustrated in the following diagram
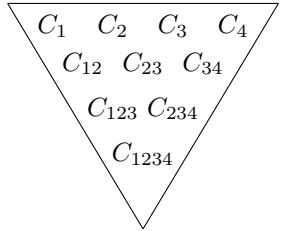
\begin{center}
 \begin{tikzpicture}
 	\draw (0,0)--(3.6,0);
	\draw (3.6,0)--(1.8,-3);
	\draw (1.8,-3)--(0,0);
	\draw (0.6,-0.3) node {$C_1$};
	\draw (1.4,-0.3) node {$C_2$};
	\draw (2.2,-0.3) node {$C_3$};
	\draw (3,-0.3) node {$C_4$};
	\draw (1,-0.8) node {$C_{12}$};
	\draw (1.8,-0.8) node {$C_{23}$};
	\draw (2.6,-0.8) node {$C_{34}$};
	\draw (1.4,-1.4) node {$C_{123}$};
	\draw (2.2,-1.4) node {$C_{234}$};
	\draw (1.8,-2.0) node {$C_{1234}$};
 \end{tikzpicture}
\end{center}
The discrete realization is known for all operators in the diagram but one: $C_{34}$. Indeed, we have by Lemma \ref{spectra} and formula (\ref{TheoremLineOne}) 
\[ C_{i+1}=\kappa(0,\beta_i-\beta_{i-1}-1), \qquad C_{[i+1]}=\kappa(x_i,\beta_i) \] 
   and the operators $C_{23}$ and $C_{234}$ are given by formula (\ref{GaugedTwo}). 
 As explained by formula (\ref{MainTrick}), we introduce an intermediate basis $\{\phi_{\vec m}\}$ of normalized eigenvectors of $C_{34}$ in the space of Dunkl-harmonics. Moreover, we consider this basis to be diagonalized by $\mathscr{Y}^\phi:=\langle C_{12},C_{34}\rangle $. Such a basis exists as $C_{12}$ and $C_{34}$ commute and the eigenvalues are explicitly known in the Dunkl model according to Lemma \ref{spectra} making these operators simultaneously diagonizable. Notice that $\mathscr{Y}^\phi$ is not a labelling Abelian algebra. We want to find the connection coefficients $\bra{ \phi_{\vec m}}\ket{ \psi_{\vec \ell}}$ with $\{\psi_{\vec \ell}\}$ the basis diagonalizing $\mathscr{Y}^\psi:=\langle C_{12},C_{123}\rangle $. Consider the following diagram:
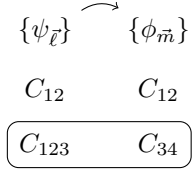
\begin{center}
\begin{tikzpicture}[->]
 \matrix(M)[matrix of nodes, column sep=0.5 cm, row sep=0.2cm]{
  $\{ \psi_{ \vec \ell}\}$ & $\{ \phi_{ \vec m}\}$\\
  $C_{12}$ 	& $C_{12}$		\\
  $C_{123}$ 	& $C_{34}$		\\
   };
 \path (M-1-1) edge[bend left] (M-1-2); 
 \node[draw=black, rounded corners=1ex, fit=(M-3-1)(M-3-2),inner sep=1pt]{};
  \end{tikzpicture}
\end{center}
In the columns of this diagram one finds each basis and the operators they diagonalize. Notice that $\mathscr{Y}_\psi$ and $\mathscr{Y}_\phi$ differ by only one element which we indicate with the box. Consider the eigenspaces of $\langle C_{12}\rangle=\mathscr{Y}_\psi\cap\mathscr{Y}_\phi$. The operators $C_{123}$ and $C_{34}$ commute with $C_{12}$ and therefore preserve these eigenspaces. The operators $C_{123}$ and $C_{34}$ are also generators of the rank one Racah algebra $R^{12,3,4}(3)$ by Lemma \ref{Racah1}:
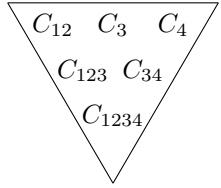
\begin{center}
 \RacahOne{C_{12}}{C_3}{C_4}{C_{123}}{C_{34}}{C_{1234}}
 \end{center}
Every element of this algebra commutes with $C_{12}$ so the algebra $R^{12,3,4}(3)$ preserves the eigenspaces of $C_{12}$. By Proposition \ref{UnivariateRacah} the connection coefficients are 
univariate Racah polynomials. To find the explicit form of these polynomials we use the isomorphism $ R^{12,3,4}(3) \cong R^{1,2,3}(3)$. This yields
\begin{align*}
   C_{12}&:                         &      \kappa(x_1,\beta_1) &=\kappa(0,\beta'_0)  \\
   C_{3}&:      &     \kappa(0,\beta_2-\beta_1-1)&=\kappa(0,\beta'_1-\beta'_0-1)\\
   C_{4}&:                    &      \kappa(0,\beta_3-\beta_2-1)&= \kappa(0,\beta'_2-\beta'_1-1) \\
   C_{123}&:                  &      \kappa(x_2,\beta_2)&= \kappa(x'_1,\beta'_1) \\
   C_{34}&:       &      \kappa(m_2,\beta_{3}-\beta_1-1)&=\kappa(k'_1,\beta'_2-\beta'_0-1) \\
   C_{1234}&:                   &      \kappa(x_{3},\beta_{3})&=\kappa(x'_2,\beta'_2).
\end{align*}
Solving this set of equations yields:
\begin{equation}\label{C34substitution}
\begin{aligned}
\beta'_0&=\beta_1+2x_1 &   \beta'_1&=\beta_2+2x_1 &   \beta'_2&=\beta_3+2x_1 \\
x'_1&=x_2-x_1 &  x'_2&=x_3-x_1 &  k'_1&=m_2.
\end{aligned}
\end{equation}
Performing these substitutions gives the appropriate connection coefficients as
\[
 \bra{ \phi_{\vec m}}\ket{\psi_{\vec \ell}}=\sqrt{\omega_{1}(\vec x')\mu_{1}(\vec k')}R_{1}(\vec k'; \vec x' ; \vec \beta';N)\delta_{m_1,\ell_1}.
\]
This also allows to write $C_{34}$ down. The following operator has the same eigenvalues and eigenvectors as $C_{34}$:
\begin{equation}\label{C34operator}
\sqrt{\omega_{1}(\vec x')}\big(-\mathcal{L}_1(0, x'_1, x'_2,\beta'_0,\beta'_1,\beta'_2,E_{x'_1})+\kappa(0,\beta'_2-\beta'_0-1)\big)\sqrt{\omega_{1}(\vec x')}^{-1}.
\end{equation}
The operator $E_{x'_1}$ equals $E_{x_2}$ as it adds $1$ to $x'_1$ but keeps the $\beta'_i$ fixed. Therefore $C_{34}$ is a three-point shift operator in the $x_2$ direction. The gauge coefficients $\omega_1(\vec x')$ differ from the gauge coefficients $\omega_2(\vec x)$ of the operators $C_{[23]}$ and $C_{[234]}$. Introduce therefore the following function:
\[ 
f:=\frac{\Gamma(\beta_{1}+x_{1})\Gamma(\beta_{1}-\beta_0+x_{1})}{\Gamma(\beta_0+x_{1}+1)x_{1}!}(\beta_1+2x_1).
\]
One checks easily that $f\omega_1(\vec x')=\omega_2(\vec x)$. Multiplying $C_{34}$ on the left by $\sqrt{f}$ and on the right by $\sqrt{f}^{-1}$ does not change the operator as $f$ does not depend on $x_2$. Finally we have the following observation. Consider the Racah operator in expression (\ref{C34operator}) and perform the substitutions (\ref{C34substitution}):
\[
 	\mathcal{L}_1(0, x_2-x_1, x_3-x_1,\beta_1+2x_1,\beta_1+2x_1,\beta_3+2x_1,E_{x_2}).
\]
This operator equals
\[
 	\mathcal{L}_1(x_1, x_2, x_3,\beta_1,\beta_2,\beta_3,E_{x_2}).
\]
This can be seen by considering the explicit expression for these Racah operators (see Definition \ref{ShiftOperator}) and observing that $B'^{p,q}_i=B^{p,q}_{i+1}$ and $b'^p_i=b^p_{i+1}$.
Hence, the operator $C_{34}$ equals
\[
 C_{34}=\sqrt{\omega_{2}(\vec x)}\big(-\mathcal{L}_1(x_1, x_2, x_3,\beta_1,\beta_2,\beta_3,E_{x_2})+\kappa(0,\beta_3-\beta_1-1)\big)\sqrt{\omega_{2}(\vec x)}^{-1}.
\]
In conclusion, we have for every generator in the algebra $R(4)$ an explicit expression. To get rid of the gauge coefficients, we conjugate every operator by $\sqrt{\omega_2(\vec x)}$. Conjugation and the lifting from the Dunkl-model preserve the algebra relations. This means the discrete operators we constructed here generate the rank $2$ Racah algebra. Moreover,  the conjugated operators $C^c_{23}$ and $C^c_{34}$ are related:
\[
 C^c_{34}=   -\mathcal{L}_1(x_1, x_2, x_3,\beta_1,\beta_2,\beta_3,E_{x_2})+\kappa(0,\beta_3-\beta_1-1)
\]
and
\[ 
C^c_{23}=-\mathcal{L}_1(x_0, x_1, x_2,\beta_0,\beta_1,\beta_2,E_{x_1})+\kappa(0,\beta_2-\beta_0-1)
\]
where we have not yet set $x_0=0$. Observe that $C^c_{34}=\sigma(C^c_{23})$. 
This proves the main Theorem \ref{MainTheorem} for the rank two Racah algebra $R(4)$.

\subsection{The general operator $C_{[i+2\ldots j+i+2]}$}
We return to the general case $R(n)$. Before we start considering the operator $C_{[i+2\ldots i+j+2]}$ with $i \geq 1$, we need to introduce the following lemma. This lemma generalizes Lemma \ref{Racah1}.
\begin{lemma}\label{Racahk}
Let $\{A_p\}_{p=1\ldots k}$ be a set of $k$ disjoint subsets of $[n]$. Define $A_B:=\cup_{q \in B} A_q$ with $B\subset [k]$. Consider the following map
\[  \theta: R(k) \rightarrow R(n) :C_{B} \rightarrow C_{A_B}. \]
This is an injective morphism. We denote its image by $R^{A_1,\ldots, A_k}(k)$ and it is isomorphic to $R(k)$
\end{lemma}
\begin{proof}
The proof of Lemma \ref{Racah1} which is given in \cite{DeBie&Genest$vandeVijver&Vinet} generalizes easily by considering a $k$-fold tensor product of $\mathfrak{su}(1,1)$ instead of a threefold tensor product mapping into the $n$-fold tensor product space.
\end{proof}
We want to find a basis of Dunkl-harmonics $\{\phi_{\vec m}\}$ so that $ \bra{ \phi_{\vec m}}\ket{\psi_{\vec \ell}}$ are multivariate Racah polynomials. This can be obtained by considering a lower rank Racah algebra inside $R(n)$ which will preserve certain eigenspaces. Then we are able to employ Theorem \ref{MultivariateRacah} to find the connection coefficients as multivariate Racah polynomials.
To this end consider the labelling Abelian algebra $\mathscr{Y}^\psi:=\langle C_{[t]}\rangle_{t=2\ldots n-1}$ which diagonalizes $\{\psi_{\vec \ell}\}$. We consider a second Abelian algebra $\mathscr{Y}^\phi$, which is not a labelling Abelian algebra, generated by the following sets
\[
\mathcal{U}_1:=\{ C_{[t]}\}_{t=2\ldots i+1}, \quad \mathcal{U}_2:= \{ C_{[t]}\}_{t=i+j+2\ldots n-1} \text{ and } \mathcal{V}_\phi:=\{ C_{[2+i \ldots i+t+2]} \}_{t=1\ldots j}. 
\]
Notice that $\mathscr{Y}^\psi \cap \mathscr{Y}^\phi=\langle \mathcal{U}_1 \cup \mathcal{U}_2 \rangle$. The operator $C_{[i+2\ldots j+i+2]}$ is in $\mathcal{V}_\phi$ so it is also in $\mathscr{Y}^\phi$. According to Lemma \ref{spectra} the operators in $\mathscr{Y}^\phi$ are diagonizable. These operators are also pairwise commutative. It follows that there exists a basis $\{ \phi_{\vec m}\}$ diagonalizing $\mathscr{Y}^\phi$. The basis $\{ \phi_{\vec m} \}$ diagonalizing $\mathscr{Y}^\phi$ necessarily also diagonalizes $C_{[i+2\ldots j+i+2]}$. We want to find the connection coefficients $\bra{ \phi_{\vec m}}\ket{ \psi_{\vec \ell}}$. Introduce the set $\mathcal{V}_\psi:=\{ C_{[t]}\}_{t=i+2\ldots i+j+1}$. The following diagram shows the change we want to make:
\begin{center}
\begin{tikzpicture}[->]
 \matrix(M)[matrix of nodes, column sep=0.5 cm, row sep=0.2cm]{
  $\{ \psi_{ \vec \ell}\}$ & $\{ \phi_{ \vec m}\}$\\
  $\mathcal{U}_1$ 	& $\mathcal{U}_1$		\\
  $\mathcal{V}_\psi$ 	& $\mathcal{V}_\phi$		\\
  $ \mathcal{U}_2$ & $\mathcal{U}_2$     \\
   };
 \path (M-1-1) edge[bend left] (M-1-2); 
 \node[draw=black, rounded corners=1ex, fit=(M-3-1)(M-3-2),inner sep=1pt]{};
  \end{tikzpicture}
\end{center}
The Abelian algebras $\mathscr{Y}^\phi$ and $\mathscr{Y}^\psi$ differ by the sets  $\mathcal{V}_\psi$ and $\mathcal{V}_\phi$ which are indicated by a box.	
Consider the eigenspaces of $\langle \mathcal{U}_1 \cup \mathcal{U}_2 \rangle=\mathscr{Y}^\psi\cap\mathscr{Y}^\phi$. The operators in the sets $\mathcal{V}_\psi$ and $\mathcal{V}_\phi$ commute with $\mathcal{U}_1 \cup \mathcal{U}_2$ and therefore preserve these eigenspaces. We have the following claim:

\begin{lemma}
The algebra $R(n)$ has a subalgebra, denoted by $R(\mathcal{V}_\psi,\mathcal{V}_\phi)$, isomorphic to a rank $j$ Racah algebra which contains the sets $\mathcal{V}_\psi$ and $\mathcal{V}_\phi$. This subalgebra commutes with  $\mathcal{U}_1 \cup \mathcal{U}_2$ and the sets $\mathcal{V}_\psi$ and $\mathcal{V}_\phi$ each generate a labelling Abelian algebra.
\end{lemma}

\begin{proof}
We want to use Lemma \ref{Racahk}. To do this, we need to find $j+2$ subsets $A_p$ of $[n]$. We choose the following:
\[ 
	A_1:=[i+1], \qquad A_p:=\{i+p\} \quad \text{ for } p\in [2\ldots j+2] .
\]
The isomorphism $\theta$ supplied by Lemma \ref{Racahk} maps $R(j+2)$ into $R(n)$, its image being $R^{[i+1],i+2,\ldots,i+j+2}(j+2)$. We denote this algebra by $R(\mathcal{V}_\psi,\mathcal{V}_\phi)$. Consider the labelling Abelian algebra $\langle C'_{[t]}\rangle_{t=2\ldots j+1}$ in $R(j+2)$ where we use an apostrophe for the generators of $R(j+2)$.  The images of its generators is given by:
\begin{equation}\label{isomorphism1}
	 \theta(C'_{[t]})=C_{A_{[t]}}=C_{[t+i]}. 
\end{equation}
These are exactly the operators in $\mathcal{V}_\psi$. Now consider the labelling Abelian algebra $\langle C'_{[2\ldots t+1]}\rangle_{t=2\ldots j+1}$ in $R(j+2)$. The images of its generators are given by:
\begin{equation}\label{isomorphism2}
	\theta(C'_{[2\ldots t+1]})=C_{A_{[2\ldots t+1]}}=C_{[2+i \ldots t+i+1]}. 
\end{equation}
These are exactly the operators in $\mathcal{V}_\phi$.
The operators in $\mathcal{U}_1=\{ C_{[t]}\}_{t=2\ldots i+1}$ commute with the subalgebra $R(\mathcal{V}_\psi,\mathcal{V}_\phi)$ by Lemma \ref{Commute} as $[t] \subset A_1$ and $[t] \cap A_p=\emptyset$ for $t \leq i+1$ and $p \geq 2$. Similarly the operators in $\mathcal{U}_2=\{ C_{[t]}\}_{t=i+j+2\ldots n-1}$ commute with the subalgebra $R(\mathcal{V}_\psi,\mathcal{V}_\phi)$ by Lemma \ref{Commute}  as $A_{[j+1]}\subset [t]$ for $t \geq j+2$.
\end{proof}

We have found an algebra $R(\mathcal{V}_\psi,\mathcal{V}_\phi)$ such that it commutes with the operators in $\mathcal{U}_1\cup\mathcal{U}_2$ and therefore preserves the eigenspaces of this set of operators. Moreover, when we restrict ourselves to these eigenspaces the bases $\{ \psi_{\vec \ell} \}$ and $\{ \phi_{\vec m}\}$ are diagonalized by the labelling Abelian algebras $\langle \mathcal{V}_\psi\rangle$ and  $\langle \mathcal{V}_\phi \rangle$ of the algebra $R(\mathcal{V}_\psi,\mathcal{V}_\phi)$.  This allows us to employ Theorem \ref{MultivariateRacah}. We conclude that the connection coefficients are multivariate Racah polynomials. 
To find the explicit expression of the connection coefficients, we use the isomorphism
\[ \theta: R(j+2) \rightarrow R(\mathcal{V}_\psi,\mathcal{V}_\phi). \]
We compare the spectra of $R(j+2)$ and $R(\mathcal{V}_\psi,\mathcal{V}_\phi)$ by Lemma \ref{spectra}.
The first set of operators we consider are the central elements:
\begin{equation*}
	\theta(C'_1)=C_{[i+1]}, \qquad \theta(C'_t)=C_{i+t} \qquad \text{ if } t \in [2 \ldots j+2] .
\end{equation*}
This leads to the following identificiation:
\begin{equation}\label{variablechange1}
	\kappa(0,\beta'_0)=\kappa(x_i,\beta_i), \qquad \kappa(0,\beta'_{t-1}-\beta'_{t-2}-1)= \kappa(0,\beta_{t+i-1}-\beta_{t+i-2}-1).
\end{equation}
The second set of operators we consider is the set $\mathcal{V}_\psi$, see (\ref{isomorphism1}). This leads to a second set of equations:
\begin{equation}\label{variablechange2}
	 \kappa(x'_{t-1},\beta'_{t-1})=\kappa(x_{t+i-1},\beta_{t+i-1}) \qquad \text{ if } t \in [2 \ldots j+1].
\end{equation}
The third set of operators we consider is the set $\mathcal{V}_\phi$, see (\ref{isomorphism2}). This leads to a third set of equations:
\begin{equation}\label{variablechange3}
	 \kappa(|\vec k'|_{t-1},\beta'_t-\beta'_0-1)=\kappa(|\vec m|_{i+1}^{i+t-1},\beta_{t+i}-\beta_{i}-1) \quad \text{ if } t \in [2\ldots j+1].
\end{equation}
with $|\vec m|_{p}^{q}=\sum_{r=p}^q m_p$.
The last equation is obtained by considering $\theta(C'_{[j+2]})=C_{[i+j+2]}$:
\begin{equation}\label{variablechange4}
       \kappa(x'_{j+1},\beta'_{j+1})=\kappa(x_{i+j+1},\beta_{i+j+1}).
\end{equation}
Solving the set of equations (\ref{variablechange1}), (\ref{variablechange2}), (\ref{variablechange3}) and (\ref{variablechange4}) for the variables with an apostrophe we find:
\begin{equation}\label{CAsubstitution}
\begin{aligned}
\beta'_t&=\beta_{t+i}+2x_{i}, \qquad &t \in [0 \ldots j+1] \\
x'_{t}&=x_{t+i}-x_{i}, \qquad &t \in [1 \ldots j] \\
k'_t&=m_{t+i}, \qquad &t \in [1 \ldots j] \\
N'&=x'_{j+1}=x_{j+i+1}-x_{i}. &
\end{aligned}
\end{equation}
Substituting these variables in the multivariate Racah polynomials and the gauge factors, one obtains the connection coefficients as:
\[
 \bra{ \phi_{\vec m}}\ket{\psi_{\vec \ell}}=\sqrt{\omega_{j}(\vec x')\mu_{j}(\vec k')}R_{j}(\vec k'; \vec x'; \vec \beta' ;N')\prod_{t=1}^{i}\delta_{m_t \ell_t}\prod_{t=i+j+1}^n\delta_{m_t \ell_t}.
\]
The eigenvector $\psi_{\vec \ell}$ of the operator $C_{[i+2 \ldots i+j+2]}$ belongs to the eigenvalue 
\[
\kappa(|\vec m|_{i+1}^{i+j},\beta_{j+i+1}-\beta_{i}-1).
\]
 We conclude that the action of $C_{[i+2 \ldots i+j+2]}$ on $ \bra{ \phi_{\vec m}}\ket{\psi_{\vec \ell}}$ coincides with the action of
\begin{equation}\label{CAoperator}
\begin{aligned}
 	\sqrt{\omega_{j}(\vec x')}\big(-\mathcal{L}_{j}(0,x'_1,\ldots,x'_{j+1},&\beta'_0,\ldots,\beta'_{j+1},E_{x'_1},\ldots,E_{x'_{j}}) \\
			&+\kappa(0,\beta'_{j+1}-\beta'_0-1)\big)\sqrt{\omega_{j}(\vec x')}^{-1}.
\end{aligned}
\end{equation}
The operator $E_{x'_t}$ equals $E_{x_{t+i}}$ for $1\leq t \leq j$ as $\vec \beta'$ and $x_i$ remain fixed. Next consider the gauge factors $\omega_j(\vec x')$. We want every operator in $R(n)$ to have the same gauge factor $\omega_{n-2}(\vec x)$. We define the following factor:
\begin{align*}
  f	&:=\frac{\omega_{n-2}(\vec x)}{\omega_j(\vec x')}\\
  	&=\prod_{t=0}^{i-1} \frac{\Gamma(\beta_{t+1}+x_{t+1}+x_{t})\Gamma(\beta_{t+1}-\beta_t+x_{t+1}-x_t)}{\Gamma(\beta_t+x_{t+1}+1+x_t)(x_{t+1}-x_t)!} \prod_{t=1}^{i} (\beta_t+2x_t) \\
	&\times \prod_{t=j+i+1}^{n-2} \frac{\Gamma(\beta_{t+1}+x_{t+1}+x_{t})\Gamma(\beta_{t+1}-\beta_t+x_{t+1}-x_t)}{\Gamma(\beta_t+x_{t+1}+1+x_t)(x_{t+1}-x_t)!} \prod_{t=j+i+1}^{n-2} (\beta_t+2x_t) .	
\end{align*}
The factor $f$ does not depend on $x_{i+1},\ldots, x_{i+j}$. Therefore the operator in formula (\ref{CAoperator}) leaves $f$ invariant and we can conjugate by $\sqrt{f}$ without altering the action of 
$C_{[i+2 \ldots i+j+2]}$. 
As a final step we perform the substitution ($\ref{CAsubstitution}$) on the Racah operator
\[
 \mathcal{L}_{j}(0,x_{i+1}-x_i,\ldots,x_{j+i+1}-x_i,\beta_i+2x_i,\ldots,\beta_{j+1}+2x_i,E_{x_{i+1}},\ldots,E_{x_{i+j}}) .
\]
This operator equals
\[
 \mathcal{L}_{j}(x_i,x_{i+1},\ldots,x_{j+i+1},\beta_i,\ldots,\beta_{j+1},E_{x_{i+1}},\ldots,E_{x_{i+j}}). 
\]
This can be seen by considering the explicit expression for these Racah operators (see Definition \ref{ShiftOperator}) and by observing that $B'^{p,q}_t=B^{p,q}_{t+i}$ and $b'^p_t=b^p_{t+i}$.
Hence, the operator $C_{[i+2 \ldots i+j+2]}$ equals
\begin{equation}\label{CAfinal}
\begin{aligned}
 	\sqrt{\omega_{n-2}(\vec x)}\big(-\mathcal{L}_{j}(x_i,x_{i+1},\ldots,&x_{j+i+1},\beta_i,\ldots,\beta_{i+j+1},E_{x_{i+1}},\ldots,E_{x_{i+j}})  \\
			&+\kappa(0,\beta_{i+j+1}-\beta_i-1)\big)\sqrt{\omega_{n-2}(\vec x)}^{-1}.
\end{aligned}
\end{equation}

\subsection{Proof of the main Theorem}
We now have  an explicit expression for every generator in the algebra $R(n)$. We conjugate the algebra by a factor $\sqrt{\omega_{n-2}(\vec x)}^{-1}$. As the conjugation and the lifting from the Dunkl-model do not change the algebra relations, the resulting algebra is a discrete realization of $R(n)$.  These are the resulting operators:
\begin{equation}\label{final1}
 C^c_{j+1}=\kappa(0,\beta_j- \beta_{j-1}-1), \qquad 2\leq j\leq n
\end{equation}
and
\begin{equation}\label{final2}
 C^c_{[j]}=\kappa(x_{j-1},\beta_{j-1}), \qquad 1\leq j \leq n
\end{equation}
These operators remain unchanged by the conjugation. From formula (\ref{GaugedTwo}) we have
\begin{equation}\label{final3}
\begin{aligned}
 C^c_{[2\ldots j+2]}=-\mathcal{L}_{j}(x_0,x_{1},\ldots,x_{j+1},\beta_0,\ldots,\beta_{j+1},E_{x_{1}},\ldots,E_{x_{j}})\\+\kappa(0,\beta_{j+1}-\beta_0-1).
 \end{aligned}
\end{equation}
We have not set $x_0=0$ yet.
Finally, by formula (\ref{CAfinal}) we have
\begin{equation}\label{final4}
\begin{aligned}
C^c_{[i+2 \ldots i+j+2]}=-\mathcal{L}_{j}(x_i,x_{i+1},\ldots,x_{j+i+1},\beta_i,\ldots,\beta_{i+j+1},E_{x_{i+1}},\ldots,E_{x_{i+j}})  \\
			+\kappa(0,\beta_{i+j+1}-\beta_i-1).
\end{aligned}
\end{equation}
Notice that $\sigma^i(C^c_{[2\ldots j+2]})=C^c_{[i+2 \ldots i+j+2]}$. The operators in formula (\ref{final2}) concur with formula (\ref{OperatorOne}). The operators in formula (\ref{final3}) concur with formula (\ref{OperatorTwo}) and the operators in formula (\ref{final1}) and formula (\ref{final4}) concur with formula (\ref{OperatorThree}). This proves Theorem \ref{MainTheorem}.

 \section*{Acknowledgements}
 We thank Luc Vinet for fruitful discussions and helpful comments.
This work was supported by the Research Foundation 
Flanders (FWO) under Grant EOS 30889451. WVDV is grateful to the Fonds Professor Frans Wuytack for supporting his research.

\end{document}